\documentclass[a4paper]{amsart}
\usepackage[english]{babel}
\usepackage[utf8]{inputenc}
\usepackage[T1]{fontenc}
\usepackage{lmodern}

\usepackage[shortlabels]{enumitem}
\setlist[enumerate]{label={\arabic*.}}
\setlist[description]{font=\normalfont\itshape}

\usepackage[dvipsnames]{xcolor}
\usepackage[backref]{hyperref}
\hypersetup{
  colorlinks=true,
  linkcolor=Maroon,
  citecolor=OliveGreen
}

\usepackage{booktabs}

\usepackage{mleftright}
\mleftright

\usepackage{todonotes}

\usepackage{cleveref}
\usepackage[nopatch=footnote]{microtype}

\usepackage{setspace}
\usepackage{amsmath,amsfonts,amsthm,mathtools,commath,amssymb}
\usepackage{aliascnt}
\usepackage{graphicx}
\usepackage{tikz}
\usetikzlibrary{calc}

\def\F{\mathbf{F}}
\def\P{\mathbf{P}}
\def\Pr{\mathbf{P}}
\def\Ex{\mathbf{E}}

\def\opt{\mathrm{opt}}

\DeclarePairedDelimiter\floor{\lfloor}{\rfloor}

\DeclareMathOperator{\Mat}{M}

\newtheorem{theorem}{Theorem}[section]
\newtheorem*{theorem*}{Theorem}
\newaliascnt{lemma}{theorem}
\newtheorem{lemma}[lemma]{Lemma}
\aliascntresetthe{lemma}
\newaliascnt{proposition}{theorem}
\newtheorem{proposition}[proposition]{Proposition}
\aliascntresetthe{proposition}
\newaliascnt{corollary}{theorem}

\aliascntresetthe{corollary}
\newtheorem*{example*}{Example}

\theoremstyle{definition}
\newaliascnt{question}{theorem}

\aliascntresetthe{question}
\newaliascnt{remark}{theorem}
\newtheorem{remark}[remark]{Remark}
\aliascntresetthe{remark}
\newaliascnt{definition}{theorem}

\aliascntresetthe{definition}
\newaliascnt{example}{theorem}

\aliascntresetthe{example}
\newaliascnt{conjecture}{theorem}

\aliascntresetthe{conjecture}
\newtheorem*{conjecture*}{Conjecture}

\crefname{theorem}{theorem}{theorems}
\Crefname{theorem}{Theorem}{Theorems}
\crefname{lemma}{lemma}{lemmas}
\Crefname{lemma}{Lemma}{Lemmas}
\crefname{proposition}{proposition}{propositions}
\Crefname{proposition}{Proposition}{Propositions}
\crefname{corollary}{corollary}{corollaries}
\Crefname{corollary}{Corollary}{Corollaries}
\crefname{question}{question}{questions}
\Crefname{question}{Question}{Questions}
\crefname{remark}{remark}{remarks}
\Crefname{remark}{Remark}{Remarks}
\crefname{definition}{definition}{definitions}
\Crefname{definition}{Definition}{Definitions}
\crefname{example}{example}{examples}
\Crefname{example}{Example}{Examples}
\crefname{conjecture}{conjecture}{conjectures}
\Crefname{conjecture}{Conjecture}{Conjectures}

\title[Optimality of finite field dimension expanders]{Optimality of finite field dimension expanders}

\author{Urban Jezernik}
\address{Urban Jezernik, Faculty of Mathematics and Physics, University of Ljubljana, Jadranska 21, 1000 Ljubljana, Slovenia; Institute of Mathematics, Physics, and Mechanics, Jadranska 19, 1000 Ljubljana, Slovenia}
\email{urban.jezernik@fmf.uni-lj.si}

\author{Markus Reineke}
\address{Markus Reineke, Faculty of Mathematics, Ruhr University Bochum, 44801 Bochum, Germany}
\email{markus.reineke@rub.de}

\thanks{This research was carried out while the second named author visited the Department of Mathematics of the University of Ljubljana and benefitted from the inspiring atmosphere. UJ was supported by the Slovenian Research Agency program P1-0222 and grants J1-50001, J1-70033. MR was supported by DFG-GRK 2240.}

\date{}

\begin{document}
\baselineskip=13pt 

\begin{abstract} 
We show that random tuples of matrices over a finite field are uniform dimension expanders for every expansion coefficient below the optimal one, using elementary counting arguments. 
The same holds for the optimal coefficient along a set of dimensions of density $1$. 
By contrast, along Fibonacci dimensions random pairs of matrices fail to be optimal dimension expanders with probability bounded away from zero.
\end{abstract}

\maketitle

\section{Introduction}

Dimension expanders were introduced by Lubotzky and Zelmanov \cite{LZ}  as a linear algebra analogue of expander graphs \cite{kowalski2019expanders}. 
Let $F$ be a field and let $k\geq 2$. A $k$-tuple $T_1,\dots,T_k\in \Mat_n(F)$ of $n\times n$-matrices over $F$ is called an \emph{$\epsilon$-expander} if 
\[
\dim(U+T_1U+\dots+T_kU)\geq(1+\epsilon)\dim U
\]
for all subspaces $U\leq F^n$ of dimension at most $n/2$. 
The positive real number $\epsilon$ is called the \emph{expansion coefficient}. 
One is interested in uniform expansion, that is,  sequences of $\epsilon$-expanders as $n \to \infty$ for a fixed $\epsilon>0$. 

Existence of uniform dimension expanders was shown for base fields of characteristic zero in \cite{LZ}, and by Bourgain \cite{Bo} for finite fields, building on the approach of Dvir and Shpilka \cite{dvir2011towards}. In both cases, the proofs are explicit constructions using advanced analytic techniques. As remarked in \cite{dvir2011towards}, over finite fields a more elementary probabilistic argument shows that random tuples of matrices are dimension expanders for some small expansion coefficient. The more relevant question is thus how large the expansion coefficient can be.

In \cite{reineke2024dimension}, dimension expanders were interpreted in terms of representations of generalized Kronecker quivers. Results on general subrepresentations \cite{Sc} were used to give a new proof of existence of expanders over algebraically closed base fields, as well as to provide optimal bounds for the expansion coefficients.

In this note, we adapt the methods of \cite{reineke2024dimension} to the case of finite base fields, by replacing an argument using the Zariski topology by an elementary counting argument. 
With this approach, we show that random tuples are uniform expanders for all expansion coefficients smaller than the optimal one computed in \cite{reineke2024dimension}, see \Cref{thm1}.
There are obstructions of number-theoretic nature to adapt this argument to the optimal expansion case, for which we establish a slightly weaker result of expansion along suitable dimensions, see \Cref{thm:optimal-expansion-along-suitable-dimensions}.
These obstructions are genuine. 
Along the dimensions $n= 2 F_{2j}$ given by Fibonacci numbers, a random pair of matrices fails to be an optimal expander with probability bounded away from zero, see \Cref{thm:fibonacci}.

\section{The quiver approach to dimension expanders}

As in \cite{reineke2024dimension}, we view a $k$-tuple $T_1,\dots,T_k\in \Mat_n(F)$ as a \emph{quiver representation}
\begin{equation} \label{eq:quiver-representation}
  1_n,T_1,\dots,T_k \colon F^n\rightarrow F^n
\end{equation}
of the generalized Kronecker quiver with $k+1$ arrows. Its subrepresentations are the pairs of subspaces $U \leq W \leq F^n$ with $T_i U \leq W$ for all $i$. We call $(\dim U, \dim W)$ the dimension vector of a subrepresentation. A subspace $U$ that violates the expansion condition gives the subrepresentation $(U, U + T_1 U + \cdots + T_k U)$, and conversely $U$ can be extracted from it. Thus the tuple is an $\epsilon$-expander if and only if \eqref{eq:quiver-representation} has no subrepresentations of dimension vector $(d,e)$ with $d \leq n/2$ and $e < (1+\epsilon)d$. 

In \cite{reineke2024dimension}, the optimal uniform expansion coefficient was computed for all algebraically closed fields $F$ using methods from quiver representation theory, as follows.
For dimension vectors $(d,e), (d',e')$, the \emph{Euler form} is $\langle (d,e), (d',e') \rangle = d d' + e e' - (k+1) d e'$.
Denote by $S(d,e)\subseteq \Mat_n(F)^k$ the subset of tuples of matrices for which \eqref{eq:quiver-representation} has a subrepresentation of dimension vector $(d,e)$. Then $S(d,e)$ is Zariski-closed, and of positive codimension if and only if $\langle (d,e), (n-d, n-e) \rangle < 0$,
as follows from Schofield's theory of general subrepresentations \cite{Sc}. 
Thus an $\epsilon$-expander exists if this condition holds for all $d\leq n/2$ and $e\leq(1+\epsilon)d$. This is true for all $n$ if and only if $\epsilon\leq \epsilon_{\opt}$, where
\[
\epsilon_{\opt}=\frac{1}{2}\left(k+1-\sqrt{k^2-2k+5}\right)
\]
is the smaller root of $P(x) = x^2 - (k+1)x + (k-1)$. We write $\widehat \epsilon_{\opt} = k+1-\epsilon_{\opt}$ for the other root. Since $P(1) = -1$, we have $\epsilon_{\opt} <  1 < \widehat \epsilon_{\opt}$. Note that $\epsilon_{\opt}$ is irrational for all $k \geq 2$.

\section{Optimally expanding tuples of matrices}

We follow the same quiver strategy as above over a finite field $\F_q$. We will bound the probability of a tuple failing to be an $\epsilon$-expander by computing,  for a random tuple of matrices, the expected number of subrepresentations of dimension vector $(d,e)$ with $d \leq n/2$ and $e < (1+\epsilon)d$. This expected number is governed by the same Euler form as the codimension of $S(d,e)$ above. Let us abbreviate
\[
E(d,e) = \langle (d,e), (n-d, n-e) \rangle = d(n-d) + e(n-e) - (k+1) d (n - e).
\]

\subsection{Counting subrepresentations}

Recall that the number of $d$-dimensional subspaces of $\F_q^n$ is given by the $q$-binomial coefficient
\[
\binom{n}{d}_q
=
q^{d(n-d)}
\frac{\prod_{i=n-d+1}^{n}(1-q^{-i})}
{\prod_{i=1}^{d}(1-q^{-i})}
\leq
C_q q^{d(n-d)},
\]
where $C_q = \prod_{i=1}^\infty (1 - q^{-i})^{-1}$, a constant depending only on $q$.
If both $d$ and $n-d$ tend to infinity, then the $\leq$ above can be replaced by $\sim$.

\begin{lemma}
  \label{lem:expected-number-of-subrepresentations}
Let $T_1,\dots,T_k \in \Mat_n(\F_q)$ be independent uniformly random.
The expected number of subrepresentations of \eqref{eq:quiver-representation} of dimension vector $(d,e)$ is at most $C_q^2 q^{E(d,e)}$, and moreover asymptotic to it as $d, e-d, n-d, n-e$ all tend to infinity.
\end{lemma}
\begin{proof}
A subrepresentation of dimension vector $(d,e)$ consists of a $d$-dimensional subspace $U$ and an $e$-dimensional subspace $W$ containing $U$, with the property that $T_i U \leq W$ for all $i$. 
There are $\binom{n}{d}_q$ choices for $U$ and $\binom{n-d}{e-d}_q$ choices for $W$ containing $U$.
For any fixed $U$ and $W$, each $T_i$ restricts to a uniformly random linear map $U \to \F_q^n$, which lands in $W$ with probability $q^{-d(n-e)}$.
Multiplying these factors and taking properties of $q$-binomial coefficients into account gives the claim.
\end{proof}

Thus the probability that a random tuple is not an $\epsilon$-expander is at most
\begin{equation}
  \label{eq:union-bound}
  \sum_{d = 1}^{\lfloor n/2 \rfloor}
  \sum_{d \leq e < (1+\epsilon) d}
  C_q^2 q^{E(d,e)},
\end{equation}
by the union bound and the lemma above.
To estimate this sum, we write $e = d + r$ with $0 \leq r < \epsilon d$. The next lemma splits the exponent $E$ into two terms. The first measures how far $r/d$ is from $\epsilon_{\opt}$, and the second how far $d$ is from $n/2$.

\begin{lemma} \label{lem:E}
For $0 \leq r \leq d \leq n/2$, we have
\[
E(d,d+r)
=
-(\epsilon_{\opt} d - r)(\widehat \epsilon_{\opt} d - r)
-(n-2d)((k-1)d - r),
\]
and both terms on the right are increasing in $r$.
\end{lemma}
\begin{proof}
Write $n - d = d + (n-2d)$ in the definition of $E$. The terms without $n-2d$ add up to $-d^2 P(r/d)$,
and the terms with $n-2d$ add up to $-(n-2d)((k-1)d - r)$. Since the roots of $P$ are $\epsilon_{\opt}$ and $\widehat \epsilon_{\opt}$, we have $d^2 P(r/d) = (\epsilon_{\opt} d - r)(\widehat \epsilon_{\opt} d - r)$. The first term is increasing in $r$ because $P$ is decreasing on $[0, (\epsilon_{\opt} + \widehat \epsilon_{\opt})/2] = [0,(k+1)/2]$, which contains $r/d \leq 1$. The second is increasing in $r$ because $n - 2d \geq 0$.
\end{proof}

\subsection{Random matrices are almost optimal dimension expanders}

The following result shows that random tuples of matrices are almost optimal dimension expanders. 
This can be viewed as an analogue of Friedman's theorem \cite{friedman2008alon} that random regular graphs are almost Ramanujan expanders with high probability as the number of vertices tends to infinity. 

\begin{theorem}\label{thm1}
Let $k \geq 2$. For any $\eta > 0$, a random tuple of $k$ matrices in $\Mat_n(\F_q)$ is an $(\epsilon_{\opt} - \eta)$-dimension expander with high probability as $n \to \infty$.
\end{theorem}
\begin{proof}
Let $\epsilon = \epsilon_{\opt} - \eta$. Write $e = d + r$ with $0 \leq r < \epsilon d$ in \eqref{eq:union-bound}. As $r/d < \epsilon < \epsilon_{\opt}$, the first term in \Cref{lem:E} is at most $-P(\epsilon) d^2$, and the second term is at most $-(k-1-\epsilon)d(n-2d)$. Hence
$E(d,d+r) \leq -P(\epsilon) d^2 - (k-1-\epsilon) d(n-2d)$,
where $P(\epsilon) > 0$ and $k-1-\epsilon > 0$ since $\epsilon < \epsilon_{\opt} < 1$. 
The exponent is thus at most $-cn$ for some $c = c(k,\eta) > 0$.\footnote{If $d \leq n/4$, the second term is $\leq -(k-1-\epsilon)n/2$, and otherwise the first one is $\leq -P(\epsilon)n^2/16$.} As there are at most $n^2$ terms, the sum \eqref{eq:union-bound} is at most $C_q^2 n^2 q^{-cn}$, which tends to $0$ as $n \to \infty$.
\end{proof}

\subsection{Dense existence of optimal dimension expanders}

Random regular graphs are only almost Ramanujan. In fact, by the spectacular result of \cite{huang2024ramanujan}, about $69 \%$ of random regular graphs are Ramanujan, with a well-understood limiting distribution of the spectral gap. We prove here that for dimension expanders, optimal dimension expansion is abundant as long as the dimensions $n$ along which we are considering the expansion are \emph{suitable} in the sense of the following result.
For a real number $x$, write $\{ x \} = x - \floor{x}$ for the fractional part of $x$.

\begin{theorem} \label{thm:optimal-expansion-along-suitable-dimensions}
Let $k \geq 2$. Let $(n_i)_{i \geq 1}$ be a sequence of positive integers tending to infinity. If $n_i = 2 d_i$ is even for infinitely many $i$, assume that
\[
d_i \left\{ \epsilon_{\opt} d_i \right\} \to \infty
\]
along those indices. Then a random tuple of $k$ matrices in $\Mat_{n_i}(\F_q)$ is an $\epsilon_{\opt}$-dimension expander with high probability as $i \to \infty$.
\end{theorem}

\begin{remark}
The condition on the sequence is satisfied by a subset of integers of density $1$, for the following reason. 
Let $A = \{ d \geq 1 \mid \{ \epsilon_{\opt} d \} \geq 1/\sqrt{d} \}$. 
The condition holds along any sequence of even dimensions $n_i = 2d_i$ with $d_i \in A$. 
The complement of the set $A$ has density $0$.\footnote{Since $\epsilon_{\opt}$ is irrational, the values $\{ \epsilon_{\opt} d \}$ are equidistributed in $[0,1]$, so for each fixed $\delta > 0$ the set of $d$ with $\{ \epsilon_{\opt} d \} < \delta$ has density $\delta$. Now if $d \notin A$ and $d > \delta^{-2}$, then $\{ \epsilon_{\opt} d \} < 1/\sqrt{d} < \delta$. Hence the complement of $A$ has upper density at most $\delta$, and taking $\delta \to 0$ it thus has density $0$.}
\end{remark}

\begin{proof}
Let $\epsilon = \epsilon_{\opt}$. We again bound the sum \eqref{eq:union-bound}, in which $e$ runs from $d$ to $\floor{(1+\epsilon)d}$ as $\epsilon$ is irrational. Using \Cref{lem:E} with $r = \epsilon d$ gives
$E(d,e) \leq - (k - 1 - \epsilon) d (n - 2d)$.
For all $d < n/2$, we have $d(n-2d) \geq (n-1)/2$, so these terms contribute at most $C_q^2 n^2 q^{-c n}$ for some $c = c(k) > 0$, which tends to $0$ as $n \to \infty$. This covers the case of odd $n$, and for even $n$ it remains to consider the terms with $d = n/2$. Write $e = d + r$ with $0 \leq r \leq \lfloor \epsilon d \rfloor$. By \Cref{lem:E}, the exponent is
\[
E(d,d+r) 
=
- (\epsilon d - r) \cdot (\widehat \epsilon d - r)
\leq
- \left( \{ \epsilon d \} + (\lfloor \epsilon d \rfloor - r) \right) 
\cdot
( \widehat \epsilon - \epsilon ) d.
\]
Writing $t = \lfloor \epsilon d \rfloor - r \geq 0$, these terms altogether thus contribute at most
\[
C_q^2 \sum_{t \geq 0} q^{- ( \hat \epsilon - \epsilon ) d (\{ \epsilon d \} + t)}
=
C_q^2 q^{- ( \hat \epsilon - \epsilon ) d \{ \epsilon d \}} \frac{1}{1 - q^{- ( \hat \epsilon - \epsilon ) d}},
\]
which tends to $0$ since $d \{ \epsilon d \} \to \infty$ by our assumption on the even dimensions.
\end{proof}

\begin{remark}
The condition $d\{\epsilon_{\opt} d\} \to \infty$ fails exactly along sequences of denominators $d$ of unusually good rational approximations of $\epsilon_{\opt}$, which are essentially the convergents of its continued fraction.\footnote{If $d_j \{\epsilon_{\opt} d_j \} \leq C$, then $r_j = \lfloor \epsilon_{\opt} d_j \rfloor$ satisfies $|\epsilon_{\opt} - r_j/d_j| = \{ \epsilon_{\opt} d_j \}/d_j \leq C/d_j^2$. By Legendre's theorem, if $C < 1/2$, then $r_j/d_j$ is a convergent of the continued fraction of $\epsilon_{\opt}$.}
\end{remark}

\subsection{Non-optimal expansion along Fibonacci dimensions}
\label{sec:fibonacci}

We now show that in cases where $d$ is a denominator of a convergent of the continued fraction of $\epsilon_{\opt}$, optimality can genuinely fail. 
We do this for $k = 2$, where $P(x) = x^2 - 3x + 1$ has roots $\epsilon_{\opt} = (3 - \sqrt{5})/2$ and $\widehat \epsilon_{\opt} = \epsilon_{\opt}^{-1}$. This is related to the golden ratio and thus to Fibonacci numbers $F_j$ with $F_0 = 0, F_1 = 1$.

\begin{theorem} \label{thm:fibonacci}
Let $k = 2$, let $n_j = 2 F_{2j}$, and let $p_j$ be the probability that a random pair in $\Mat_{n_j}(\F_q)$ is an $\epsilon_{\opt}$-dimension expander. Then
\[
\frac{\lambda_q}{1 + \lambda_q}
\leq
\liminf_{j \to \infty} (1 - p_j)
\leq
\limsup_{j \to \infty} (1 - p_j)
\leq
\lambda_q,
\qquad
\text{where }
\lambda_q = C_q^2 q^{-1}.
\]
\end{theorem}

For $q = 2$, we have $C_2 \approx 3.463$, $\lambda_2 \approx 5.995$, so a random pair fails to be an optimal expander with probability at least about $0.857$. 
For large $q$, both bounds are of order $1/q$, so failure of optimality is a phenomenon of small finite fields.

Let us write $d = F_{2j}$, $n = 2d$, and $r = F_{2j-2}$. 
The exponent in \eqref{eq:union-bound} of the corresponding term with $e = d + r$ is $E(d,d+r) = -(d^2 + r^2 - 3 dr) = -1$ by \Cref{lem:E}.\footnote{Here we used Cassini's identity $dr - F_{2j-1}^2 = -1$ and the recursion $F_{2j-1} + r = d$.} In fact, this is precisely the critical term $r = \floor{\epsilon_{\opt} d}$.\footnote{Since $(\epsilon_{\opt} d - r)(\epsilon_{\opt}^{-1} d - r) = 1$ by \Cref{lem:E} and the second factor is larger than $1$, the first factor must be strictly between $0$ and $1$, and thus $r = \floor{\epsilon_{\opt} d}$.} A subrepresentation $(U,W)$ of dimension vector $(d, d+r)$ thus witnesses the failure of optimal expansion. There can be several such subrepresentations that arise from a common subspace $W$ of dimension $d+r$. The largest possible $U$ for which $(U,W)$ is a subrepresentation is the subspace
\[
L_W = \{ w \in W \mid T_1 w, T_2 w \in W \}.
\]
We call the subspace $W$ of dimension $d+r$ a \emph{witness} for the failure of optimal expansion if $\dim L_W \geq d$. For uniformly random $T_1, T_2 \in \Mat_n(\F_q)$, let $X$ be the number of witnesses, so that $p_j \leq \P(X = 0)$.
We estimate this probability by comparing $X$ to the number $X_c$ of subrepresentations of dimension vector $(c, d+r)$ for $c \in \{ d, d+1 \}$.

\begin{lemma}
We have $0 \leq X_d - X \leq C_q q^d X_{d+1}$.
\end{lemma}
\begin{proof}
Let $W$ run over all subspaces of dimension $d+r$, and let $\kappa = \dim L_W$. Then
\[
X = \sum_W \textbf{1}_{\kappa \geq d}
\qquad
\text{and}
\qquad
X_c = \sum_W \binom{\kappa}{c}_q
\quad (c \in \{ d, d+1 \}).
\]
Let us inspect $X_d - X$. Its terms with $\kappa \leq d$ contribute the same to both sums. Each term with $\kappa > d$ can be bounded by
\[
\binom{\kappa}{d}_q
=
\frac{q^{d+1} - 1}{q^{\kappa - d} - 1} \binom{\kappa}{d+1}_q
\leq
\binom{d+1}{d}_q \binom{\kappa}{d+1}_q
\leq
C_q q^d \binom{\kappa}{d+1}_q.
\]
Summing over all $W$ gives $0 \leq X_d - X \leq C_q q^d X_{d+1}$.
\end{proof}

\begin{proposition}[first moment]
$\Ex X \to \lambda_q$ as $j \to \infty$.
\end{proposition}
\begin{proof}
\Cref{lem:expected-number-of-subrepresentations} gives $\Ex X_d \sim C_q^2 q^{E(d,d+r)} = \lambda_q$
and $\Ex X_{d+1} \ll_q q^{-1.5d}$.\footnote{We have $E(d,d+r) = -1$ and $E(d+1,d+r) - E(d,d+r) = -1 - 3(d-r) \leq -1 - 1.5d$.}
The preceding lemma now implies $\Ex (X - \lambda_q) \to 0$, as claimed.
\end{proof}

For the second moment, we need the two lemmas about pairs of distinct witnesses $W, W'$. Two subspaces of dimension $d+r$ in $\F_q^{2d}$ intersect in dimension at least $2r$, and we write $\dim(W \cap W') = 2r + w$ with $w \geq 0$.

\begin{lemma} \label{lem:expected-number-of-pairs}
The expected number of pairs $(U,W), (U',W')$ counted by $X_d$ with $W \neq W'$, $\dim(U \cap U') = u$ and $\dim(W \cap W') = 2r + w$ is
\[
\ll_q q^{- u^2 + w (3 u - w - 2 r)}.
\]
If $w = 0$, the events that $(U,W)$ and $(U',W')$ are subrepresentations are independent.
\end{lemma}
\begin{proof}
Let $A = U \cap U'$, $B = W \cap W'$, and write $U = A \oplus U_1$, $U' = A \oplus U_1'$. The pairs $(U,W)$, $(U',W')$ are subrepresentations if and only if $T_i A \leq B$, $T_i U_1 \leq W$, $T_i U_1' \leq W'$ for $i = 1,2$. The sum $A + U_1 + U_1'$ is direct, so these conditions restrict disjoint parts of $T_i$ and are independent. Their probabilities are $q^{-2a}$, $q^{-2b}$, $q^{-2b}$, where $a = u(2d - 2r - w)$ and $b = (d-u)(d-r)$. For $w = 0$, the product of these probabilities has exponent $-4d(d-r)$, which is twice the exponent $-2d(d-r)$ for a single pair $(U,W)$, so the two events are independent. The expected number in question is the sum of these probabilities over all pairs $(U,W)$, $(U',W')$ of subspaces with the given dimensions, and each summand equals $q^{-2a-4b}$. The number of pairs is at most the number of ways to choose $B$, then $A \leq B$, and then $W \supseteq B$ and $U$ with $A \leq U \leq W$, and likewise $W'$ and $U'$.
Counting these choices and multiplying by $q^{-2a-4b}$ gives the claimed bound.\footnote{The count is $\ll_q q^{c}$ with $c
= 
(2r+w)(2d-2r-w) 
+ u(2r+w-u) 
+ 2( (d-r-w)(d-r) + (d-u)r )$.
Note that since $d^2 + r^2 - 3dr = 1$, we have $c - 2a - 4b = -u^2 + w(3u - w - 2r) - 2$.}
\end{proof}

\begin{lemma} \label{lem:dimension-of-intersection}
Let $W \neq W'$ be witnesses, and let $c = d - r - w$ be the codimension of $W \cap W'$ in $W$. Then $\dim(L_W \cap L_{W'}) \geq d - 3c$.
\end{lemma}
\begin{proof}
Since $L_{W'} \leq W'$ and $W \cap W'$ has codimension $c$ in $W'$, the subspace $A = L_{W'} \cap W$ has dimension at least $\dim L_{W'} - c \geq d - c$. For $x \in A$ we have $T_i x \in W'$, so the linear map $x \mapsto (T_1 x + W, T_2 x + W)$ maps $A$ into $((W + W')/W)^2$, which has dimension $2c$. Its kernel $A \cap L_W$ thus has dimension at least $d-3c$.
\end{proof}

\begin{proposition}[second moment] \label{prop:second-moment}
$\Ex X(X-1) \leq \lambda_q^2 + o(1)$ as $j \to \infty$.
\end{proposition}
\begin{proof}
The expected value $\Ex X(X-1)$ is the sum of $\Pr(W \text{ and } W' \text{ are witnesses})$ over all pairs of distinct subspaces $W, W'$ of dimension $d+r$. 
If $W,W'$ are witnesses, then \Cref{lem:dimension-of-intersection} gives subrepresentations $(U,W)$, $(U',W')$ counted by $X_d$ with $\dim(U \cap U') \geq d - 3c$.
Hence $\Ex X(X-1)$ is at most the expected number of such pairs, which we group according to $u = \dim(U \cap U')$ and $w = \dim(W \cap W') - 2r$, and estimate each group with \Cref{lem:expected-number-of-pairs}.

If $w = 0$, the two events are independent, so those pairs contribute at most $(\Ex X_d)^2$, which tends to $\lambda_q^2$. For $w \geq 1$, the pairs with given $u$ and $w$ contribute $\ll_q q^{G}$, where $G = -u^2 + w(3u - w - 2r)$. 
We have $G \leq 1 - \delta d$ for some $\delta > 0$ and all large $d$.\footnote{$G$ is concave in $u$ with maximum at $3w/2$, and $u \geq d - 3c = 3w - (2d - 3r)$. If $3w/2 \geq d - 3c$, then $G \leq -w(27r - 10d)/6$, else $G \leq 1 - c(4d - 9r + c)$. Now use $w, c \geq 1$ and $10/27 < \epsilon_{\opt} < 4/9$.}
Now sum over the pairs $(u,w)$, the total contribution is $\ll_q d^2 q^{- \delta d} = o(1)$ and the claim follows.
\end{proof}

\begin{proof}[Proof of \Cref{thm:fibonacci}]
The two moment calculations and Cauchy-Schwarz give
\[
1 - p_j \geq 
\P(X \geq 1) \geq
\frac{(\Ex X)^2}{\Ex (X^2)} \geq
\frac{\lambda_q^2 + o(1)}{\lambda_q^2 + \lambda_q + o(1)} =
\frac{\lambda_q}{1 + \lambda_q} + o(1).
\]
For the upper bound, recall from the proof of \Cref{thm:optimal-expansion-along-suitable-dimensions} that subrepresentations violating expansion with $\dim U < d$ exist with probability $o(1)$. Any other violation is a subrepresentation $(U,W)$ with $\dim U = d$ and $\dim W \leq d + r$, so there must exist a witness $W' \geq W$. Hence $1 - p_j \leq \P(X \geq 1) + o(1) \leq \Ex X + o(1) \to \lambda_q$.
\end{proof}

\bibliographystyle{alpha}
\bibliography{biblio_offde}

\end{document}